\documentclass[11pt]{amsart}
\usepackage[francais,english]{babel}
\usepackage{palatino}
\usepackage{amsfonts}
\usepackage{amssymb}
\usepackage[arrow, matrix, curve]{xy}
\usepackage{xypic}
\usepackage{amscd}
\usepackage[breaklinks,bookmarksopen,bookmarksnumbered]{hyperref}
\usepackage[dvips]{graphicx}
\usepackage{color}

\newtheorem{theorem}{Theorem}[section]
\newtheorem{proposition}[theorem]{Proposition}

\newtheorem{lemma}[theorem]{Lemma}
\theoremstyle{definition}
\newtheorem{definition}[theorem]{Definition}
\newtheorem{remark}[theorem]{Remark}

\newtheorem{conjecture/question}[theorem]{Conjecture/Question}

\newtheorem{remark/definition}[theorem]{Remark/Definition}
\newtheorem{terminology/notation}[theorem]{Terminology/Notation}

\def\GG{{\textbf G}}

\def\PP{{\textbf P}}
\def\TT{{\textbf T}}
\def\OO{\mathcal{O}}

\def\cA{\mathcal{A}}

\def\P{\mathcal{P}}

\def\I{\mathcal{I}}

\def\cM{\mathcal{M}}
\def\cR{\mathcal{R}}
\def\rr{\overline{\mathcal{R}}}
\def\cZ{\mathcal{Z}}

\def\cU{\mathcal{U}}
\def\cY{\mathcal{Y}}

\def\cX{\mathcal{X}}

\def\Pic0{{\rm Pic}^0(X)}

\def\mm{\overline{\mathcal{M}}}

\def\cc{\overline{\mathcal{C}}}
\def\px{\widetilde{\mathcal{X}}}
\def\py{\widetilde{\mathcal{Y}}}
\def\pr{\widetilde{\mathcal{R}}}
\def\pc{\widetilde{\mathcal{C}}}

\newcommand{\dblq}{{/\!/}}

\def\aa{\overline{\mathcal{A}}}
\def\pa{\widetilde{\cA}}
\def\rr{\overline{\mathcal{R}}}

\begin{document}
\title{The universal abelian variety over $\cA_5$}

\author[G. Farkas]{Gavril Farkas}

\address{Humboldt-Universit\"at zu Berlin, Institut F\"ur Mathematik,  Unter den Linden 6
\hfill \newline\texttt{}
 \indent 10099 Berlin, Germany} \email{{\tt farkas@math.hu-berlin.de}}
\thanks{}

\author[A. Verra]{Alessandro Verra}
\address{Universit\'a Roma Tre, Dipartimento di Matematica, Largo San Leonardo Murialdo \hfill
\indent 1-00146 Roma, Italy}
 \email{{\tt
verra@mat.uniroma3.it}}
\maketitle
\begin{abstract}
{We establish a structure result for the universal abelian variety over $\cA_5$. This implies that the boundary divisor of $\aa_6$ is unirational and leads to a lower bound on the slope of the cone of effective divisors on $\aa_6$.}
\end{abstract}
\vskip 3pt

The general principally polarized abelian variety $[A, \Theta]\in \cA_g$ of dimension $g\leq 5$ can be realized as a Prym variety. Abelian varieties of small dimension can be studied in this way via the rich and concrete theory of curves, in particular, one can establish that $\cA_g$ is unirational in this range. In the case $g=5$, the Prym map $P:\cR_6\rightarrow \cA_5$ is finite of degree $27$, see \cite{DS}; three different proofs \cite{Don}, \cite{MM}, \cite{Ve1} of the unirationality of $\cR_6$ are known. The moduli space $\cA_g$ is of general type for $g\geq 7$, see \cite{Fr}, \cite{Mu}, \cite{T}. Determining the Kodaira dimension of $\cA_6$ is a notorious open problem.

The aim of this paper is to give a simple unirational parametrization of the universal abelian variety over $\cA_5$ and hence of the boundary divisor of a compactification of $\cA_6$. We denote by $\phi:\cX_{g-1}\rightarrow \cA_{g-1}$ the universal abelian variety of dimension $g-1$ (in the sense of stacks). The moduli space $\pa_g$ of principally polarized abelian varieties of dimension $g$ and their rank $1$ degenerations is a partial compactification of $\cA_g$ obtained by blowing-up $\cA_{g-1}$ in the Satake compactification, cf. \cite{Mu}. Its boundary $\partial \pa_g$ is isomorphic to the universal Kummer variety in dimension $g-1$ and there exist a surjective double covering $j: \cX_{g-1}\rightarrow \partial \pa_g$. We establish a simple structure result for the boundary $\partial \pa_6$:
\begin{theorem}\label{unir1}
The universal abelian variety $\cX_5$ is unirational.
\end{theorem}
This immediately implies that the boundary divisor $\partial \pa_6$ is unirational as well. What we prove is actually stronger than Theorem \ref{unir1}. Over the moduli space $\cR_g$ of smooth Prym curves of genus $g$, we consider the universal Prym variety $\varphi:\cY_g\rightarrow \cR_g$ obtained by pulling-back $\cX_{g-1}\rightarrow \cA_{g-1}$ via the Prym map $P:\cR_g\rightarrow \cA_{g-1}$. Let $\rr_g$ be the moduli space of stable Prym curves of genus $g$ together with the universal Prym curve $\tilde{\pi}:\pc\rightarrow \rr_g$ of genus $2g-1$. If $\pc^{g-1}:=\pc \times_{\rr_g}\ldots \times_{\rr_g} \pc$ is the $(g-1)$-fold product, one has a universal \emph{Abel-Prym} rational map $\mathfrak{ap}:\pc^{g-1}\dashrightarrow \cY_{g}$, whose restriction on each individual Prym variety is the usual Abel-Prym map. The rational map $\mathfrak{ap}$ is dominant and generically finite (see Section 4 for details). We prove the following result:

\begin{theorem}\label{unir2}
The five-fold product $\pc^5$ of the universal Prym curve over $\rr_6$ is unirational.
\end{theorem}

The key idea in the proof of Theorem \ref{unir2} is to view smooth Prym curves of genus $6$ as discriminants of conic bundles, via their representation as symmetric determinants of quadratic forms in three variables.  We fix four general points $u_1, \ldots, u_4\in \PP^2$ and set $w_i:=(u_i, u_i)\in \PP^2\times \PP^2$. Since the action of the automorphism group $\mbox{Aut}(\PP^2\times \PP^2)$ on $\PP^2\times \PP^2$ is $4$-transitive, any set of four general points in $\PP^2\times \PP^2$ can be brought to this form. We then consider the linear system
$$\PP^{15}:=\Bigl|\I_{\{w_1, \ldots, w_4\}}^2(2,2)\Bigr|\subset \Bigl|\OO_{\PP^2\times \PP^2}(2,2)\Bigr|$$
of hypersurfaces $Q \subset \PP^2\times \PP^2$ of bidegree $(2,2)$ which are nodal at $w_1, \ldots, w_4$. For a general threefold $Q\in \PP^{15}$, the first projection
$p:Q\rightarrow \PP^2$ induces a conic bundle structure with a sextic discriminant curve $\Gamma\subset \PP^2$ such that $p(\mbox{Sing}(Q))=\mbox{Sing}(\Gamma)$. The discriminant curve $\Gamma$ is nodal precisely at the points $u_1, \ldots, u_4$. Furthermore, $\Gamma$ is equipped with an unramified double cover $p_{\Gamma}:\widetilde{\Gamma}\rightarrow \Gamma$, parametrizing the lines which
are components of the singular fibres of $p:Q\rightarrow \PP^2$. By normalizing, $p_{\Gamma}$ lifts to an unramified double cover
$f:\widetilde C \to C$ between the normalization $\widetilde C$ of $\widetilde \Gamma$ and the normalization $C$ of $\Gamma$ respectively.
Note that there exists an exact sequence of generalized Prym varieties
$$ 0 \longrightarrow (\mathbf C^*)^4 \longrightarrow P(\widetilde{\Gamma}/\Gamma) \longrightarrow P(\widetilde{C}/C) \longrightarrow 0.$$
One can show without much effort that the assignment $\PP^{15}\ni Q\mapsto [\widetilde{C}\stackrel{f}\rightarrow C]\in \cR_6$ is dominant. This offers an alternative, much simpler, proof of the unirationality of $\cR_6$. However, much more can be obtained with this construction.

\vskip 4pt

Let $\GG:=\PP^2\times (\PP^2)^{\vee}=\bigl\{(o, \ell):o\in \PP^2, \ell\in \{o\}\times (\PP^2)^{\vee}\bigr\}$ be the Hilbert scheme of lines in the fibres of the first projection $p:\PP^2\times \PP^2\rightarrow \PP^2$. Since containing a given line in a fibre of $p$ imposes \emph{three} linear conditions on the linear system $\PP^{15}$ of threefolds $Q \subset \PP^2\times \PP^2$ as above, it follows
that imposing the condition $\{o_i\}\times \ell_i\subset Q$ for \emph{five} general lines, singles out a \emph{unique} conic bundle $Q\in \PP^{15}$. This induces an \'etale double cover
$f:\widetilde{C}\rightarrow C$, as above, over a smooth curve of genus $6$. Moreover, $f$ comes equipped with five marked points $\ell_1, \ldots, \ell_5\in \widetilde{C}$. To summarize, we can define a rational map
$$\zeta:\GG^5\dashrightarrow \pc^5, \ \ \ \zeta\Bigl((o_1, \ell_1), \ldots, (o_5, \ell_5)\Bigr):=\Bigl(f:\widetilde{C}\rightarrow C, \ell_1, \ldots, \ell_5\Bigr), $$
between two 20-dimensional varieties, where $\GG^5$ denotes the $5$-fold product of $\GG$.

\begin{theorem}\label{dominance}
The morphism $\zeta:\GG^5\dashrightarrow \pc^5$ is dominant, so that $\pc^5$ is unirational.
\end{theorem}

More precisely, we show that $\GG^5$ is birationally isomorphic to the fibre product $\PP^{15}\times_{\rr_6} \pc^5$.
In order to set Theorem \ref{dominance} on the right footing and in view of enumerative calculations, we introduce a $\PP^2$-bundle $\pi:\PP(\cM)\rightarrow S$ over the quintic del Pezzo surface
$S$ obtained by blowing-up $\PP^2$ at the points $u_1, \ldots, u_4$.
The rank $3$ vector bundle $\cM$ on $S$ is obtained by making an elementary transformation along the four exceptional divisors $E_1, \ldots, E_4$ over $u_1, \dots, u_4$. The nodal threefolds $Q\subset \PP^2\times \PP^2$ considered above can be thought of as sections of a tautological linear system on $\PP(\cM)$, and via the identification
$$\Bigl|\I_{\{w_1, \ldots, w_4\}}^2(2,2)\Bigr|=\Bigl|\OO_{\PP(\cM)}(2)\Bigr|,$$ we can view $4$-nodal conic bundles in $\PP^2\times \PP^2$ as \emph{smooth} conic bundles over $S$. In this way the numerical characters of a pencil of such conic bundles can be computed (see Sections 2 and 3 for details).

\vskip 5pt

Theorem \ref{dominance} is then used to give a lower bound for the slope of the effective cone of $\aa_6$ (though we stop short of determining the Kodaira dimension of $\aa_6$).
Recall that if $E$ is an effective divisor on the perfect cone compactification $\aa_g$ of $\cA_g$ with no component supported on the boundary $D_g:=\aa_g-\cA_g$ and $[E]= a\lambda_1 -b[D_g]$, where $\lambda_1\in CH^1(\pa_g)$ is the Hodge class, then the slope of $E$ is defined as $s(E):=\frac{a}{b}\geq 0$. The slope $s(\aa_g)$ of the effective cone of divisors of $\aa_g$ is the infimum of the slopes of all effective divisors on $\aa_g$. This important invariant governs to a large extent the birational geometry of $\cA_g$; for instance, since $K_{\aa_g}=(g+1)\lambda_1-[D_g]$, the variety $\aa_g$ is of general type if $s(\aa_g)<g+1$, and uniruled when $s(\aa_g)>g+1$. It is shown in the appendix of \cite{GSMH} that the slope of the moduli space $\cA_g$ is independent of the choice of a toroidal compactification.

It is known that $s(\aa_4)=8$ and that the Jacobian locus $\mm_4\subset \aa_4$ achieves the minimal slope \cite{SM}; one of the results of \cite{FGSMV} is the calculation $s(\aa_5)=\frac{54}{7}$. Furthermore, the only irreducible effective divisor on $\aa_5$ of minimal slope is the closure of the Andreotti-Mayer divisor $N_0'$
consisting of $5$-dimensional ppav $[A, \Theta]$ for which the theta divisor $\Theta$ is singular at a point which is not $2$-torsion. Concerning $\aa_6$, we establish the following estimate:
\begin{theorem}\label{slope}
The following lower bound holds: $s(\aa_6)\geq \frac{53}{10}.$
\end{theorem}
Note that this is the first concrete lower bound on the slope of $\aa_6$. The idea of proof of Theorem \ref{slope} is very simple. Since $\pc^5$ is unirational, we choose a suitable sweeping rational curve $i:\PP^1\rightarrow \pc^5$, which we then push forward to $\aa_6$ as follows:
$$\xymatrix{
    \PP^1 \ar[r]^i  \ar@/^2pc/[rrrr]^h &  \pc^5  \ar[r]^{\mathfrak{ap}}  &  \py_6  \ar[r]  &  \px_5  \ar[r]^j & D_6 \\
   }
$$
Here $\py_6$ and $\px_5$ are partial compactifications of $\cY_6$ and $\cX_5$ respectively which are described in Section 4, whereas $D_6$ is the boundary divisor of $\aa_6$. The curve $h(\PP^1)$ sweeps the boundary divisor of $\aa_6$ and intersects non-negatively any effective divisor on $\aa_6$ not containing $D_6$. In particular, $$s(\aa_6)\geq \frac{h(\PP^1)\cdot [D_6]}{h(\PP^1)\cdot \lambda_1}.$$
To define $i:\PP^1\rightarrow \pc^5$, we fix general points $(o_1, \ell_1), \ldots, (o_4, \ell_4)\in \GG$ and a further general point $o\in \PP^2$. Then we consider the image under $\zeta$ of the pencil of lines in $\PP^2$ through $o$, that is, the sweeping curve $i$ is defined as
$$\PP(T_o(\PP^2))\ni \ell\mapsto \zeta\Bigl((o_1,\ell_1), \ldots, (o_4, \ell_4), (o, \ell)\Bigr)\in \pc^5.$$
The calculation of the numerical characters of $h(R)\subset \aa_6$ is a consequence of the geometry of the map $\zeta$ and is completed in Section 4.

\vskip 4pt

We close the Introduction by discussing the structure of the paper. Theorem \ref{dominance} (and hence also Theorems \ref{unir1} and \ref{unir2}) are proven rather quickly in  Section 1. The bulk of the paper is devoted to the explicit description of the numerical characters of a curve that sweeps the boundary divisor of $\aa_6$ and to the proof of Theorem \ref{slope}. Section 2 concerns enumerative properties of pencils of conic bundles over quintic del Pezzo surfaces. The sweeping curve for the boundary divisor of $\cA_6$ is constructed in Section 3. In the final Section 4, we prove Theorem \ref{slope}.

\noindent {\bf{Acknowledgment:}} We are grateful to Sam Grushevsky for explaining to us important aspects of the geometry of the universal Kummer variety (especially Proposition \ref{mumford}). We also thank the referee for a careful reading and for suggestions that improved the presentation.

\section{Determinantal nodal sextics and a parametrization of $\cX_5$}
In this section we prove Theorem \ref{dominance}. We begin by recalling basic facts about determinantal representation of nodal plane sextics, see \cite{B2}, \cite{Dol}, \cite{DIM}. Let $\Gamma\subset \PP^2$ be an integral $4$-nodal sextic  and $\nu:C\rightarrow \Gamma$ the normalization map, thus $C$ is a smooth curve of genus $6$. One has an exact sequence at the level of $2$-torsion groups
$$0\longrightarrow \mathbb Z_2^{\oplus 4}\longrightarrow \mbox{Pic}^0(\Gamma)[2]\stackrel{\nu^*}\longrightarrow \mbox{Pic}^0(C)[2]\longrightarrow 0.$$
In particular, \'etale double covers $f:\Gamma'\rightarrow \Gamma$ with an irreducible source curve $\Gamma'$ are induced by $2$-torsion points $\eta\in \mbox{Pic}^0(\Gamma)[2]$, such that
$\eta_C:=\nu^*(\eta)\neq \OO_C$.

\begin{definition}\label{p6}
We denote by $\P_6$ the quasi-projective moduli space of pairs $(\Gamma, \eta)$ as above, where $\Gamma\subset \PP^2$ is an integral $4$-nodal sextic and $\eta\in \mathrm{Pic}^0(\Gamma)[2]$ is a torsion point such that $\eta_C\neq \OO_C$. Equivalently, the induced double cover $\Gamma'\rightarrow \Gamma$ is \emph{unsplit}, that is, the curve $\Gamma'$ is irreducible. 
\end{definition}

Starting with a general element $[C, \eta_C]\in \cR_6$, since $|W^2_6(C)|=5$, there are five sextic nodal plane models $\nu:C\rightarrow \Gamma$. For each of them, there are $2^4$ further ways of choosing $\eta\in (\nu^*)^{-1}(\eta_C).$ Thus there is a degree $80=5\cdot 2^4$ covering $\rho:\P_6\rightarrow \cR_6$.
\vskip 3pt

Suppose now that $(\Gamma, \eta)\in \P_6$ is a general point\footnote{We shall soon establish that $\P_6$ is irreducible, but here we just require that $\rho(\Gamma, \eta)$ be a general point of the irreducible variety $\cR_6$.}. In particular $h^0(\Gamma, \eta(1))=0$, or equivalently, $h^0(\Gamma, \eta(2))=3$. Indeed, the condition
$h^0(\Gamma, \eta(1))\geq 1$ implies that $\Gamma\subset \PP^2$ possesses a totally tangent conic, that is, there exists a reduced conic $B \subset \PP^2$ such that $\nu^*(B) = 2b$,
with $b$ being an effective divisor of $C$. This condition is satisfied only if $\rho(\Gamma, \eta)$ lies in the ramification divisor of the Prym map $P:\cR_6\rightarrow \cA_5$, see \cite{FGSMV}. Thus we may assume that $h^0(\Gamma, \eta(2))=3$, for a general point $(\Gamma, \eta)\in \P_6$.

Following \cite{B2} Theorem B, it is known that such a sheaf $\eta$ admits a resolution
\begin{equation}\label{res}
0\longrightarrow \OO_{\PP^2}(-4)^{\oplus 3}\stackrel{A}\longrightarrow \OO_{\PP^2}(-2)^{\oplus 3}\longrightarrow \eta\longrightarrow 0,
\end{equation}
where the map $A$ is given by a symmetric matrix $\Bigl(a_{ij}(x_1, x_2, x_3)\Bigr)_{i, j=1}^3$ of quadratic forms. More precisely, we can view the resolution
(\ref{res}) as a twist of the exact sequence
\begin{equation}\label{res2}
0\longrightarrow H^0(\Gamma, \eta(2))^{\vee}\otimes \OO_{\PP^2}(-2) \stackrel{A}\longrightarrow H^0(\Gamma, \eta(2))\otimes \OO_{\PP^2} \stackrel{\mathrm{ev}}\longrightarrow \eta(2)\longrightarrow 0,
\end{equation}
where $\mathrm{ev}$ is the evaluation map on sections. Indeed, since the multiplication map $H^0(\Gamma, \eta(2))\otimes H^0(\OO_{\PP^2}(j))\rightarrow H^0(\Gamma, \eta(2+j))$ is surjective for all $j\in \mathbb Z$ (use again that $H^0(\Gamma, \eta(1))=0$, see also \cite{Ve2} Proposition 3.1 for a very similar situation), it follows that the kernel of the morphism $\mbox{ev}$ splits as a sum of line bundles, which then necessarily must be $\OO_{\PP^2}(-2)^{\oplus 3}$.

Since $\eta$ is invertible, for each point $x\in \Gamma$ one has
$$1=\mbox{dim}_{\mathbb C}\ \eta(x)=3-\mbox{rk } A(x),$$
where, as usual, $\eta(x):=\eta_x\otimes_{\OO_{\Gamma, x}} \mathbb C(x)$ is the fibre of the sheaf $\eta$ at the point $x$. Thus $\mbox{rk } A(x)=2$, for each $x\in \Gamma$.

To the matrix $A\in M_3\bigl(H^0(\OO_{\PP^2}(2))\bigr)$ we can associate the following $(2, 2)$ threefold in $\PP^2_{[x_1:x_2:x_3]}\times \PP^2_{[y_1:y_2:y_3]}=\PP^2\times \PP^2$
$$Q: \sum_{i, j=1}^3 a_{ij}(x_1, x_2, x_3)y_i y_j=0,$$ which is a conic bundle with respect to the two projections. Alternatively, if
$$A: H^0(\Gamma, \eta(2))^{\vee}\otimes H^0(\Gamma, \eta(2))^{\vee}\rightarrow H^0(\Gamma, \OO_{\Gamma}(2))$$
is the symmetric map appearing in (\ref{res2}), then $A$ induces the $(2,2)$ hypersurface $$Q\subset  \PP \Bigl (H^0(\Gamma, \OO_{\Gamma}(1))^{\vee}\Bigr)\times \PP \Bigl(H^0(\Gamma, \eta(2))^{\vee}\Bigr)=\PP^2\times \PP^2.$$
We denote by $p:Q\rightarrow \PP^2$ the first projection and then $\Gamma\subset \PP^2$ is precisely the discriminant curve of $Q$ given by determinantal equation $\Gamma:=\bigl\{\mbox{det } A(x_1, x_2, x_3)=0\bigr\}$. Let $\Gamma'$ denote the Fano scheme of lines $F_1\bigl(p^{-1}(\Gamma)/\Gamma\bigr)$ over the discriminant curve $\Gamma$. That means that  $\Gamma'$ parametrizes pairs $(x, \ell)$, where $x\in \Gamma$ and $\ell$ is an irreducible component of the fibre $p^{-1}(x)$. The map $f:\Gamma'\rightarrow \Gamma$ is given by $f(x, \ell):=x$. Since $\mbox{rk } A(x)=2$ for all $x\in \Gamma$, it follows that $f$ is an \'etale double cover.
\begin{proposition}
For a general point $(\Gamma, \eta)\in \P_6$, the restriction map $p_{|\mathrm{Sing}(Q)}:\mathrm{Sing}(Q)\rightarrow \mathrm{Sing}(\Gamma)$ is bijective.
\end{proposition}
\begin{proof}
Let $x\in \Gamma$ and $R:=\OO_{\PP^2,x}$ be the local ring of $\PP^2$  and $\mathfrak m$ its maximal ideal. After a linear change of coordinates, we may assume that the matrix
$A \mbox{ mod } \mathfrak m=:A(x)$ equals
$$A(x)=\begin{pmatrix}
1 & 0& 0\\
0 & 1 & 0 \\
0 & 0 & 0\\
\end{pmatrix}
$$
Suppose $(x, y=[y_1, y_2, y_3])\in \mathrm{Sing}(Q)$. Then $A(x)\cdot ^ty=0$, hence $y_1=y_2=0$. Imposing that the partials of the defining equation of $Q$ with respects to $x_1, x_2, x_3$ vanish, we obtain that $a_{33}\in \mathfrak m^2$. Since $\mbox{det}(a_{ij})\equiv a_{33} \mbox{ mod } \mathfrak m^2$, this implies that $\Gamma$ is singular at $x$. Conversely, for $x\in \mbox{Sing}(\Gamma)$, we obtain that $\mbox{Sing}(Q)\cap p^{-1}(x)=\{(x,y)\}$, where $y\in \PP^2$ is uniquely determined by the condition $A(x)\cdot ^t y=0$ (use once more that $\mbox{rk } A(x)=2$).
\end{proof}

\begin{proposition}
$f_*(\OO_{\Gamma'})=\OO_{\Gamma}\oplus \eta$, that is, the double cover $f$ is induced by $\eta$.
\end{proposition}
\begin{proof} Essentially identical to \cite{B1} Lemme 6.14.
\end{proof}

To sum up, to a  general point $(\Gamma, \eta)\in \P_6$ we have associated a $4$-nodal conic bundle $p:Q\rightarrow \PP^2$ as above. Conversely, as explained in the Introduction, the discriminant curve of  a $4$-nodal conic bundle in $Q\subset \PP^2\times \PP^2$ gives rise to an element in $\P_6$.    

Let $\TT\subset \Bigl|\OO_{\PP^2\times \PP^2}(2,2)\Bigr|$ be the subvariety consisting of $4$-nodal hypersurfaces of bidegree $(2,2)$. This is an irreducible $31$-dimensional variety endowed with an action of $\mbox{Aut}(\PP^2\times \PP^2)$. The following result summarizes what has been achieved so far:
\begin{theorem}\label{r6thm}
A general Prym curve $(\Gamma, \eta)\in \P_6$ is the discriminant of a $4$-nodal conic bundle $p:Q\rightarrow \PP^2$, where $Q\subset \PP^2\times \PP^2$ is a $4$-nodal threefold of bidegree $(2,2)$. More precisely, we have a birational isomorphism $\TT\dblq \mathrm{Aut}(\PP^2\times \PP^2)\stackrel{\cong}\dashrightarrow \P_6$.
\end{theorem}
\begin{remark}
A similar isomorphism between the moduli space of Prym curves over \emph{smooth} plane sextics and the quotient $\Bigl|\OO_{\PP^2\times \PP^2}(2,2)\Bigr|\dblq \mbox{Aut}(\PP^2\times \PP^2)$ has already been established and used in  \cite{B1} and \cite{Ve2}.
\end{remark}
\begin{remark}
Theorem \ref{r6thm} yields another (shorter) proof of the unirationality of $\cR_6$.
\end{remark}
\vskip 3pt
The automorphism group of $\PP^2\times \PP^2$ sits in an exact sequence
$$0\longrightarrow PGL(3)\times PGL(3)\longrightarrow \mbox{Aut}(\PP^2 \times \PP^2)\longrightarrow \mathbb Z_2\longrightarrow 0.$$
In particular, we can fix four general points $u_1, \ldots, u_4\in \PP^2$, as well as diagonal points $w_i:=(u_i, u_i)\in \PP^2\times \PP^2$, and consider the
linear system $\PP^{15}:=\Bigl|\I_{\{w_1, \ldots, w_4\}}^2 (2,2)\Bigr|$ of $(2,2)$ threefolds with assigned nodes at these points. Theorem \ref{r6thm} implies the existence of a dominant discriminant
map $\mathfrak{d}:\PP^{15}\dashrightarrow \P_6$ assigning $\mathfrak{d}(Q):=(\Gamma'\stackrel{f}\rightarrow \Gamma)$.

\vskip 4pt

\noindent \emph{Proof of Theorem \ref{dominance}.}
Using the notation introduced in this section and in the Introduction, setting $\mu:=\varphi\circ \mathfrak{ap}: \pc_5\dashrightarrow \rr_6$, one has the following commutative diagram:

$$
     \xymatrix{
         \GG^5 \ar[r] \ar[d] & \P_6\times_{\cR_6}\pc^5 \ar[d] \ar[r] & \pc^5 \ar[d]_{\mu} \\
          \PP^{15} \ar[r]^{\mathfrak{d}}       & \P_6 \ar[r]^{\rho} & \rr_6 }
$$

The dominance of the composite map $\zeta:\GG^5\dashrightarrow \pc^5$ follows once we observe,  that the above diagram is birationally a fibre product, that is,
$\GG^5\stackrel{\cong}\dashrightarrow \PP^{15}\times_{\rr_6} \pc^5.$
\hfill $\Box$

\section{Conic bundles over a del Pezzo surface}
With view to further applications, we analyze the linear system of conic bundles of type $(2, 2)$ in $\PP^2 \times \PP^2$ which are singular at four fixed general points and birationally, we reconstruct such a linear system as the complete linear system of smooth conic bundles in a certain $\mathbf P^2$-bundle over a smooth quintic del Pezzo surface.

We fix four general points $u_1, \ldots, u_4\in \PP^2$ and set $w_i:=(u_i, u_i)\in \PP^2\times \PP^2$.  Let $S$ be the Del Pezzo surface defined by the blow-up
$\sigma:S \rightarrow \PP^2$ of $u_1, \ldots, u_4$.  For $i=1, \ldots, 4$, we denote by $E_i:=\sigma^{-1}(u_i)$ the exceptional line over $u_i$. Set $E:=E_1+\cdots+E_4$ and denote by $L \in |\sigma^* \mathcal O_{\mathbf P^2}(1)|$ the pull-back of a line in $\mathbf P^2$ under $\sigma$.  An important role is played by the rank $3$ vector bundle $\cM$ on $S$ defined by the following sequence
\begin{equation}\label{seqM}
0 \longrightarrow \cM \stackrel{ j} \longrightarrow  H^0(S, \OO_S(L))\otimes \OO_S(L)\stackrel{r} \longrightarrow \bigoplus_{i=1}^4 \OO_{E_i}({  2L})\longrightarrow 0.
\end{equation}
Here $r_i: H^0(S,\mathcal O_S(L)) \otimes \mathcal O_S(L) \to \mathcal O_{E_i}(2L)$ is the evaluation map and $r := \oplus_{i = 1}^4 r_i$. Since $\mathcal O_{E_i}(L)$ is trivial, it follows that $h^0(r)$ is surjective. Passing to cohomology, we write the exact sequence
$$ 0 \longrightarrow H^0(S, \cM) \stackrel {h^0(j)} \longrightarrow H^0(  S, O_S(L))\otimes H^0(S, \OO_S(L)) \stackrel{h^0(r)}  \longrightarrow \bigoplus_{i=1}^4 H^0(\OO_{E_i}(2L)) \longrightarrow 0. $$
In particular, we obtain that $h^0(S, \cM) = 5$.  By direct calculation, we also find that
\begin{equation}\label{chV}
c_1(\cM)=\OO_S(-K_S) \ \ \mbox{ and } \ \  c_2(\cM)=3.
\end{equation}
Under the decomposition
$ H^0(\mathcal O_S(L)) \otimes H^0(\mathcal O_S(L)) = \wedge^2 H^0(\mathcal O_S(L)) \oplus H^0(\mathcal O_S(2L))$
into symmetric and anti-symmetric tensors, the space $H^0(S, \mathcal M) \subset H^0(\OO_S(L))\otimes H^0(\OO_S(L))$ decomposes as
$$
H^0(S, \mathcal M) = H^0(S, \mathcal M)^- \oplus H^0(S, \mathcal M)^+ = \bigwedge^2 H^0(S, \mathcal O_S(L)) \oplus H^0(S, \OO_S(2L - E)).
$$

\begin{lemma}\label{globgen}
The vector bundle $\cM$ is globally generated.
\end{lemma}
\begin{proof} %The map of stalks $j_x: \mathcal M_x \to H^0(S, \mathcal O_{S}(L)) \otimes \mathcal O_{S,x}(L)$ is an isomorphism for each $x \in U := S - \bigcup_{i = 1}^4 E_i$.
%Furthermore, it is easy to see that the subspaces
%$$
%V^-(-x) \subset \bigwedge^2 H^0\bigl(S, \mathcal O_S(L)\bigr)  \mbox{ and }  \ V^+(-x) \subset H^0\bigl(S, \mathcal O_S(2L - \sum_{i=1}^4 E_i)\bigr),
%$$
%have both dimension one. Since $h^0(S, \mathcal M) = 5$ and $j_x$ is an isomorphism, it follows that the evaluation map
%$H^0(S, \mathcal M) \to \mathcal M_x$ is surjective.

Clearly, we only need to address the global generation of $\mathcal M$ along $\bigcup_{i=1}^4 E_i$ and to that end, we consider the restriction of the sequence (\ref{seqM}) to $E_i$,
$$
 \cM_{| E_i}  \stackrel {j_{|E_i}} \longrightarrow  H^0(S, \OO_S(L))\otimes \OO_{E_i} \stackrel{r_{|E_i}} \longrightarrow  \OO_{E_i} \longrightarrow 0,
$$
where we recall that $\OO_{E_i}(L)$ is trivial. The sheaf $H^0(\mathcal O_S(L-E_i)) \otimes \mathcal O_{E_i} = \mathcal O_{\mathbf P^1}^{\oplus 2}$ is the kernel of $r_{| E_i}$. Since $\mbox{det}(\mathcal M_{| E_i}) = \mathcal O_{E_i}(1)$, it follows that $\mathcal M_{|E_i}$ fits into an exact sequence of bundles on $\PP^1$:
$$
0 \longrightarrow \mathcal O_{\mathbf P^1}(1) \longrightarrow \mathcal M_{| E_i}  \stackrel {j|E_i} \longrightarrow \mathcal O_{\mathbf P^1}^{\oplus 2} \longrightarrow 0.
$$
This sequence is split, so that $\mathcal M_{| E_i} = \mathcal O_{\mathbf P^1}(1) \oplus \mathcal O_{\mathbf P^1}^{\oplus 2}$, which is globally generated.
The same holds for $\mathcal M$ if the map $H^0(\mathcal M) \to H^0(\mathcal M_{| E_i})$ is surjective; this is implied by the vanishing $H^1(S, \mathcal M(-E_i))=0$.
We twist by $\mathcal O_S(-E_i)$ the sequence (\ref{seqM}), and write
$$
0 \longrightarrow \cM( - E_i) \longrightarrow  H^0(S, \OO_S(L))\otimes \OO_S(L - E_i)\stackrel{r} \longrightarrow \bigoplus_{i=1}^4 \OO_{E_i}{  (2L - E_i })\longrightarrow 0.
$$
Since $h^0(r)$ is surjective and $h^1(S, \OO_S(L-E_i)) = 0$, it follows $h^1(S, \cM(- E_i)) = 0$. \end{proof}

\vskip 3pt

From now on we set $\PP:=\PP(\cM)$ and consider the $\mathbf P^2$-bundle $\pi: \PP \rightarrow S$.  The linear system $|\OO_{\PP}(1)|$ is base point free, for $\mathcal M$ is globally generated. We reserve the notation
$$
h:=\phi_{\OO_{\PP}(1)}: \PP\rightarrow \PP^4 := \PP H^0(S, \mathcal M)^{\vee}.
$$
for the induced morphism.
A Chern classes count implies  that $\mbox{deg}(h)=2$. The map $j$ from the sequence (\ref{seqM}) induces a birational morphism $$\epsilon:S \times \PP^2 \dashrightarrow \PP.$$
We describe a factorization of $\epsilon$. Since $j$ is an isomorphism along $U := S - \bigcup_{i = 1}^4 E_i$, it follows that $\epsilon: U \times \mathbf P^2 \to \pi^{-1}(U)$ is biregular.
The behaviour of $\epsilon$ along $E_i \times \mathbf P^2$ can be understood in terms of the restriction of the sequence (\ref{seqM}) to $E_i$. Following the proof of Lemma \ref{globgen}, one has the exact sequence
$$
0 \longrightarrow \mathcal O_{\PP^1}(1) \longrightarrow  \cM_{| E_i}  \stackrel {j_{| E_i}}\longrightarrow  H^0(S, \OO_S(L))\otimes \OO_{E_i} \stackrel{r_{| E_i}} \longrightarrow  \OO_{E_i} \longrightarrow 0,
$$
where $\mbox{Im}(j_{| E_i}) = H^0(\OO_S(L - E_i)) \otimes  \OO_{E_i}$. Now $j_{| E_i}$ induces a rational map
$$
\epsilon_{|E_i\times \PP^2}:E_i \times \PP^2  \dashrightarrow \mathbf P (\mathcal M_{| E_i}) \subset \PP.
$$
For a point $x \in E_i$, the restriction of $\epsilon$ to $\mathbf P^2 \times \lbrace x \rbrace $ is the projection  $\{x\}\times \PP^2 \to \mathbf P^1$ of center $(x, u_i)$. This implies that:

\begin{lemma} The birational map $\epsilon$ contracts $E_i\times \PP^2$ to a surface which is a copy of $\mathbf P^1 \times \mathbf P^1$.  Furthermore, the indeterminacy scheme of $\epsilon$ is equal to $\bigcup_{i=1}^4 E_i \times  \{u_i\}$.
\end{lemma}
Let $D_i := E_i \times \lbrace u_i \rbrace\subset S\times \PP^2$ and $D :=D_1+\cdots+D_4$. We consider the blow-up $$\alpha: \widetilde {S \times \mathbf P^2} \to S \times \mathbf P^2 $$
of $S\times \PP^2$ along $D$, and the birational map
$$
\epsilon_2 := \epsilon \circ \alpha: \widetilde{S \times \mathbf P^2} \to \PP.
$$
The restriction of $\epsilon_2$ to the strict transform $\widetilde {E_i \times \mathbf P^2}$ of $E_i \times \mathbf P^2$ is a regular morphism, for $\epsilon_{|E_i \times \mathbf P^2}$ is defined by the linear system $\bigl|\mathcal I_{E_i\times \{u_i\}/S\times \PP^2}(1, 1)\bigr|$. This implies that $\epsilon_2$ itself is a regular morphism:

 \begin{proposition}
The following commutative diagram resolves the indeterminacy locus of $\epsilon$:
$$\xymatrix{ &  \widetilde{S\times \PP^2} \ar[dl]_{\alpha} \ar[dr]^{\epsilon_2} & \\
   S\times \PP^2 \ar@{.>}[rr]^{\epsilon}  &  &  {\PP}       \\ }$$
\end{proposition}
\noindent In the sequel, it will be useful to consider  the exact commutative diagram

$$
     \xymatrix{
         H^0(S,\mathcal M) \ar[r] \ar[d] & H^0(\OO_S(L))\otimes H^0(\OO_S(L)) \ar[d] \ar[r]^{} & \bigoplus_{i = 1}^4 H^0( \mathcal O_{E_i}(2L)) \ar[d] \\
          H^0\bigl(\mathcal I_{\{w_1 \dots w_4\}}(1,1)\bigr) \ar[r]^{}       & H^0(\OO_{\PP^2}(1))\otimes H^0(\OO_{\PP^2}(1)) \ar[r]^{} & \bigoplus_{i = 1}^4 H^0(\mathcal O_{w_i}(2)) }
$$

where the vertical arrows are isomorphisms induced by $\sigma: S \to \mathbf P^2$. Starting from the left arrow, one can construct the commutative diagram
$$
     \xymatrix{
         H^0(S,\cM)\otimes \OO_S \ar[r] \ar[d] & \cM \ar[d]^{j} \\
          H^0\bigl(\mathcal I_{\{w_1 \dots w_4\}} (1,1)\bigr)\otimes \OO_S \ar[r]^{}       & H^0(S,L)\otimes \OO_S(L) }
$$

%$$
%\begin{CD}
% {H^0(S, \mathcal M) \otimes \mathcal O_S}@>>> {\mathcal M}   \\
%@VVV @VjVV   \\
% {H^0\bigl(\mathcal I_{\{w_1 \dots w_4\}} (1,1)\bigr)\otimes \OO_S \otimes \mathcal O_S} @>{(\sigma \times \mathrm{id}_{\mathbf P^2})^*}>> {H^0(S,L) \otimes \mathcal O_S(L)}.   \\
%\end{CD}

Passing to evaluation maps, we obtain the morphism
$h: \PP \to \mathbf P^4$ and the rational map $h_D: S \times \mathbf P^2 \dashrightarrow \mathbf P^4$ defined by the space
$\bigl(\sigma \times \mathrm{id}_{\mathbf P^2}\bigr)^* H^0\Bigl(\PP^2\times \PP^2, \mathcal I_{\{w_1, \ldots, w_4\}}(1,1)\Bigr)$.  \par
 %where $\mathcal I_D$ is the ideal sheaf of $D$ and $\mathcal I_D(m,n)$ denotes the tensor product with $(\sigma \times id_{\mathbf P^2})^* \mathcal O_{\mathbf P^2 \times \mathbf P^2}(1,1)$.
 The discussion above is summarized in the following commutative diagram:
  $$\xymatrix{ &  {\widetilde{S \times \mathbf P^2} } \ar[dl]_{\alpha} \ar[dr]^{\epsilon_2} & \\
    {S \times \mathbf P^2} \ar@{.>}[rr]^{\epsilon} \ar@{.>}[dr]_{h_D} & &   { \PP} \ar[dl]^h &  \\
  & {\mathbf P^4} }
 $$

 We derive a few consequences. Let $\pi_1: S \times \mathbf P^2 \to S$ and $\pi_2: S \times \mathbf P^2 \to \mathbf P^2$
be the two projections, then define the following effective divisors of $\widetilde {S \times \mathbf P^2}$:
 $$
\widetilde H  \in \bigl|(\pi_1 \circ \alpha)^* (\mathcal O_S(-K_S))\bigr| \ , \ \widetilde H_1 \in \bigl|(\pi_1 \circ \alpha)^*(\mathcal O_S(L))\bigr| \ , \ \widetilde H_2 \in \bigl|(\pi_2 \circ \alpha)^*(\mathcal O_{\mathbf P^2}(1))\bigr|,
$$
as well as
$$
\widetilde N_i := \alpha^{-1}(D_i) \ \mbox{ and } \ \widetilde N = \sum_{ i = 1}^4 \widetilde N_i.
$$
Applying push-forward under $\epsilon_{2}$, we obtain the following divisors on $\PP$:
$$
H:=\epsilon_{2*}(\widetilde H),  \ \  H_i:=\epsilon_{2*}(\widetilde H_i),  \ \ N_i:=\epsilon_{2*}(\widetilde N_i),  \mbox{ and } \  \ N:= \sum_{i=1}^4 N_i.
$$
\begin{proposition} $|\OO_{\PP}(1)|= |H_1 + H_2 - N|$. \end{proposition}
\begin{proof} Using for instance \cite{Ma} Theorem 1.4, we have $\epsilon_2^*(\OO_{\PP}(1))=\OO_{\widetilde{S\times \PP^2}}(\widetilde H_1 + \widetilde H_2 - \widetilde N)$.  By pushing forward, we obtain the desired result.
%This implies $\epsilon_2^*M \equiv \widetilde H_1 + \widetilde H_2 - \widetilde N$, and then the conclusion follows.
\end{proof}

We have already remarked that $h: \PP \to \mathbf P^4$ is a morphism of degree $2$. The inverse image $E \subset \PP$ under $h$ of a general line in $\PP^4$ is a smooth elliptic curve.
The restriction $h_{E}$ has $4$ branch points and the branch locus of $h$ is a quartic hypersurface $B \subset \PP^4$.

\begin{proposition} For each $d\geq 0$, one has $h^0(\PP, \mathcal O_{\PP}(d)) = \binom{d+4}4 + \binom{2d}4$.
\end{proposition}
\begin{proof} We pass to the Stein factorization $h := s \circ f$, where $f: \overline {\PP} \to \mathbf P^4$ is a double cover and
$s: \PP \to \overline{\PP}$ is birational. In particular,  $h^0(\PP, \mathcal O_{\PP}(d)) = h^0(f^* \mathcal O_{\PP^4}(d))$.  The involution $\iota: \overline{\PP} \to \overline{\PP}$ induced by $f$ acts on $H^0(f^*\mathcal O_{\PP^4}(d))$ and  the eigenspaces are $f^*H^0(\mathcal O_{\PP^4}(d))$ and
$b \cdot f^* H^0(\mathcal O_{\mathbf P^4}(2d-4))$ respectively, where $b \in H^0(f^*\mathcal O_{\PP^4}(2))$ and $\mbox{div}(b) = f^{-1}(B)$.
\end{proof}

We can now relate the $15$-dimensional linear system $|\OO_{\PP}(2)|$ of \emph{smooth} conic bundles in $\PP$ to the linear system of $4$-\emph{nodal} conic bundles of type $(2,2)$ in $\mathbf P^2 \times \mathbf P^2$.
%The latter linear system is $\vert \mathcal I_{w_1 \dots w_4}(2,2) \vert$.
 %An easy computation shows that $\mbox{dim }  |\mathcal I_{w_1 \dots w_4}(2,2)| $ $= 15 =$ $ \dim \vert 2M \vert$. \par
Let $\tilde I$ be the moving part of  the total transform $\bigl((\sigma \times \mathrm{id}_{\mathbf P^2}) \circ \alpha\bigr)^*\Bigl|\mathcal I^2_{\{w_1, \ldots, w_4\}}(2,2)\Bigr|$.
Over $\PP$, we consider the linear system $I':= (\epsilon_{2})_*\tilde I$, and conclude that:
\begin{proposition} One has the equality $I' =|\OO_{\PP}(2)|$ of linear systems on $\PP$.
\end{proposition}
\begin{proof} Consider a general threefold $Y \in \bigl|\mathcal I_{\{w_1, \ldots, w_4\}}(1,1)\bigr|$.  Its strict transform $\widetilde{Y}$  under the morphism $(\sigma \times \mathrm{id}_{\mathbf P^2}) \circ \alpha$ is smooth and has class
$\widetilde H_1 + \widetilde H_2 -\widetilde N$. Therefore we obtain  $(\epsilon_{2})_*(\widetilde Y)\in \vert H_1 + H_2 - N \vert = |\OO_{\PP}(1)|$, and then $I' = \vert \OO_{\PP}(2)\vert$.
\end{proof}

To conclude this discussion, the identification $$\bigl| \mathcal O_{\PP}(2) \bigr| = \Bigl| \mathcal I^2_{\{w_1, \ldots w_2\}}(2,2) \Bigr| := \PP^{15},
$$
induced by the birational map $\epsilon$, will be used throughout the rest of the paper.

\begin{remark} One can describe $h: \PP \to \mathbf P^4$ in geometric terms. Consider the rational map
$h':=h_D \circ \bigl(\sigma \times \mathrm{id}_{\PP^2}\bigr)^{-1}:\PP^2 \times \PP^2 \dashrightarrow \PP^4$, where $h_D$ appears in a previous diagram. Then $h'$ is defined by the linear system $\bigl|\mathcal I_{\{w_1, \ldots, w_4\}}(1,1)\bigr|$. If $\mathbf P^2 \times \mathbf P^2 \subset \mathbf P^8$
is the Segre embedding and $\Lambda \subset \mathbf P^8$ the linear span of $w_1, \ldots, w_4$, then $h'$
is the restriction to $\mathbf P^2 \times \mathbf P^2$ of the linear projection having center $\Lambda$.
\vskip 3pt

One can also recover the tautological bundle $\mathcal M$ as follows. Consider the family of planes $\Bigl\{\PP_x := h'_*\bigl(\{x \} \times \mathbf P^2\bigr)\Bigr\}_{x\in \PP^2}$. Its closure  in the
Grassmannian $\GG(2,4)$ of planes of $\PP^4$ is equal to the image of $S$ under the classifying map of $\mathcal M$. We omit further details.
 \end{remark}

\begin{proposition}\label{intnumb}
The following relations hold in $CH^4(\widetilde{S\times \PP^2})$:
$$\widetilde{N}^4=-4, \ \widetilde{N}^3\cdot \widetilde{H}=4, \  \widetilde{N}^3\cdot \widetilde{H}_1=\widetilde{N}^3\cdot \widetilde{H}_2=0, \ \widetilde{N}^2\cdot \widetilde{H}^2=\widetilde{N}^2\cdot \widetilde{H}_1^2=\widetilde{N}^2\cdot \widetilde{H}_2^2=0.$$
\end{proposition}
\begin{proof}
These are standard calculations on blow-ups. We fix $i\in \{1, \ldots, 4\}$ and note that $\widetilde{N}_i=\PP\bigl(\OO_{D_i}^{\oplus 2}\oplus \OO_{D_i}(1)\bigr)$. We denote by $\xi_i:=c_1(\OO_{\widetilde{N}_i}(1))\in CH^1(\widetilde{N}_i)$ the class of the tautological bundle on the exceptional divisor, by $\alpha_i:=\alpha_{| \widetilde{N}_i}:\widetilde{N}_i\rightarrow D_i$ the restriction of $\alpha$, and by $j_i:\widetilde{N}_i\hookrightarrow \widetilde{S\times \PP^2}$ the inclusion map.
Then for $k=1, \ldots, 4$, the formula $\widetilde{N}_i^k=(-1)^{k-1} (j_i)_*(\xi_i^{k-1})$ holds in $CH^k(\widetilde{S\times \PP^2})$. In particular, $$\widetilde{N}_i^4=-(j_i)_*(\xi_i^3)=-c_1(\OO_{D_i}^{\oplus 2}\oplus \OO_{D_i}(1))=-1,$$ which implies that $\widetilde{N}^4=\widetilde{N}_1^4+\cdots +\widetilde{N}_4^4=-4$.
Furthermore, based on dimension reasons, $\widetilde{N}_i^2\cdot \alpha^*(\gamma)=-(j_i)_*\bigl(\xi_i\cdot \alpha_i^*(\gamma_{| D_i})\bigr)=0$, for each class $\gamma\in CH^2(S\times \PP^2)$. Finally, for a class $\gamma\in CH^1(S\times \PP^2)$, we have that
$\widetilde{N_i}^3\cdot \alpha^*(\gamma)=(j_i)_*\bigl(\xi_i^2\cdot \alpha_i^*(\gamma_{| D_i})\bigr)=(\alpha_i)_*(\xi_i^2)\cdot \gamma_{| D_i}=\gamma\cdot D_i$, where the last intersection product is computed on $S\times \PP^2$. This  determines all top intersection numbers involving $\widetilde{N}^3$, which finishes the proof.
%The normal bundle $N_{D_i / S \times \lbrace u_i \rbrace}$ is $\mathcal O_{E_i}(E_i)$, while the normal bundle $N_{S \times \lbrace u_i \rbrace / S \times \mathbf P^2}$ is trivial.  Then the standard exact sequence of normal bundles
%$$
%0 \to N_{D_i / S \times \lbrace u_i \rbrace} \to N_{D_i/ S \times \mathbf P^2} \to  N_{S \times \lbrace u_i \rbrace / S \times \mathbf P^2\vert D_i }\to 0
%$$
%implies ${\widetilde N_i}^4 = \deg(N_{D_i / S \times \mathbf P^2}) = -1$. Since $\widetilde N_i \cap \tilde N_j = \emptyset$ for $ i \neq j$, it follows $\widetilde N^4 = -4$. To prove that
%$\widetilde N_i^3 \cdot \widetilde H = 1$ recall that $D_i \cdot H = 1$ and $\widetilde H = \alpha^*H$. Since $\widetilde N_i$ is the projectivized normal bundle of $D_i$ and $\widetilde N_i^3$ is the
%class of a tautological curvilinear section of it, we have $\widetilde N_i^3 \cdot \widetilde H = 1$ and $\widetilde N \cdot \widetilde H = 4$. Notice also that $D_i \cdot H^2 = 0$ implies $\widetilde N^2 \cdot \widetilde H^2 = 0$. Finally $\widetilde H_1$ and $\widetilde H_2$ contract $\tilde N_i$ to a point, which leads to the other equalities.

\end{proof}
\begin{remark} Since $\epsilon_2$ contracts the divisors $\widetilde{E_i\times \PP^2}$, clearly $H=3H_1-N$. An immediate consequence of Proposition \ref{intnumb} is that the degree of the morphism $h:\PP\rightarrow \PP^4$ equals
$\mbox{deg}(h)=(H_1+H_2-N)^4=6H_1^2\cdot H_2^2+N^4=2$.
\end{remark}

 \subsection{Pencils of conic bundles in the projective bundle $\PP$}
 In this section we determine the numerical characters of a pencil of $4$-nodal conic bundle of type $(2, 2)$. Let
$$P\subset \bigl|\OO_{\PP}(2)\bigr|=\bigl|\OO_{\PP}(2H_1+2H_2-2N)\bigr|$$ be a Lefschetz pencil in $\PP$. We may assume that its base locus $B\subset \PP$ is a smooth surface. We are primarily interested in the number of singular conic bundles and those having a double line respectively. We first describe $B$.

\begin{lemma}\label{ints}
For the base surface $B\subset \PP$ of a pencil of conic bundles, the following hold:
\begin{enumerate}
\item $K_B=\OO_B(H_1+H_2-N)\in \mathrm{Pic}(B).$
\item $K_B^2=8$ and $c_2(B)=64$.
\end{enumerate}
\end{lemma}
\begin{proof} The surface $B$ is a complete intersection in $\PP$, hence by adjunction
$$K_B=K_{\PP| B}\otimes \OO_B(4H_1+4H_2-4N).$$
Furthermore, $K_{\widetilde{S\times \PP^2}}=\alpha^* \bigl(\OO_S(-H)\boxtimes \OO_{\PP^2}(-3)\bigr)\otimes \OO_{\widetilde{S\times \PP^2}}(2\widetilde{N})$, and by push-pull
$$K_{\PP}=(\epsilon_2)_*\bigl(K_{\widetilde{S\times \PP^2}}\bigr)=\OO_{\PP}(-H-3H_2+2N)=\OO_{\PP}(-3H_1-3H_2+3N),$$
for $H=3H_1-N$. We find that $K_B=\OO_B(H_1+H_2-N)$. From Lemma \ref{intnumb}, we compute
$$K_B^2=4(H_1+H_2-N)^2\cdot (H_1+H_2-N)^2=24H_1^2\cdot H_2^2+4N^4=8.$$
Finally, from the Euler formula applied for $B$, we obtain $12\chi(B, \OO_B)=K_B^2+c_2(B)$. Since $\chi(B, \OO_B)=6$, this yields $c_2(B)=64$.
\end{proof}

For a variety $Z$ we denote as usual by $e(Z)$ its topological Euler characteristic.
\begin{lemma}\label{eul1}
For a general conic bundle $Q\in \bigl|\OO_{\PP}(2)\bigr|$, we have that $e(Q)=4$, whereas for conic bundle $Q_0$ with a single ordinary quadratic singularity, $e(Q_0)=5$.
\end{lemma}
\begin{proof}
We fix a conic bundle $\pi_1:Q\rightarrow S$ with smooth discriminant curve $C\in |-2K_S|$. We then write the relation
$e(Q-\pi_1^*(C))=2e(S-C)$. Since $e(\pi_1^*(C))=3e(C)$, we find that $e(Q)=2e(S)+e(C)=2\cdot 7-10=4$.

Similarly, if $\pi_1:Q_0\rightarrow \PP^2$ is a conic bundle such that
the discriminant curve $C_0\subset S$ has a unique node, then $e(Q_0)=2e(S)+e(C_0)=14-9=5$.
\end{proof}

In the next statement we use the notation from \cite{FL} for divisors classes on $\rr_g$, see also the beginning of Section 3 for further details.

\begin{theorem}\label{singconic}
In a Lefschetz pencil of conic bundles $P\subset |\OO_{\PP}(2)|$ there are precisely $77$ singular conic bundles and $32$ conic bundles with a double line.
\end{theorem}
\begin{proof}
Retaining the notation from above, $B\subset \PP$ is the base surface of the pencil.  The number $\delta$ of nodal conic bundles in $P$ is given by the formula:
$$ \delta=e(\PP)+e(B)-2e(Q) =3e(S)+64-2\cdot 4=77,$$
where the relation $e(\PP)=3e(S)$ follows because $\pi: \PP\rightarrow S$ is a $\PP^2$-bundle.

\vskip 3pt

The number of conic bundles in the pencil $P$ having a double line equals the number of discriminant curves in the family induced by $P$ in $\rr_6$, that lie in the ramification divisor $\Delta_0^{\mathrm{ram}}$ of the projection map  $\pi:\rr_6\rightarrow \mm_6$. We choose general conic bundles $Q_1, Q_2\in P$, and let $A=\bigl(a_{ij}(x_1, x_2, x_3)\bigr)_{i, j=1}^3$ and $B=\bigl(b_{ij}(x_1, x_2, x_3)\bigr)_{i,j=1}^3$ be the symmetric matrices of quadratic forms giving rise to Prym curves
$(\Gamma_1, \eta_1):=\mathfrak{d}(Q_1)$ and $(\Gamma_2, \eta_2):=\mathfrak{d}(Q_2)\in \P_6$ respectively. Note that both curves $\Gamma_1$ and $\Gamma_2$ are nodal precisely at the points $u_1, \ldots, u_4$. Let us consider the surface
$$Y:=\Bigl\{\Bigl([x_1:x_2:x_3],[t_1:t_2]\Bigr)\in \PP^2\times \PP^1: \mbox{det}\bigl((t_1 a_{ij}+t_2 b_{ij})(x_1, x_2, x_3)\bigr)=0 \Bigr\},$$
together with the projection $\gamma:Y\rightarrow \PP^1$.
If $h_1, h_2\in CH^1(\PP^2\times \PP^1)$ are the pull-backs of the hyperplane classes under the two projections, then $Y\equiv 6h_1+3h_2$.  Therefore $\omega_Y=\OO_Y(3h_1+h_2)$ and
$h^0(Y, \omega_Y)=20$. Observe that the surface $Y$ is singular along the curves $L_j:=\{u_j\}\times \PP^1$ for $j=1, \ldots, 4$, and let $\nu_Y:\mathcal{Y}\rightarrow Y$ be the normalization. From the exact sequence
$$0\longrightarrow H^0(\cY, \omega_{\cY})\longrightarrow H^0(Y, \omega_Y)\longrightarrow \bigoplus_{j=1}^4 H^0(L_j, \omega_{Y | L_j})\longrightarrow 0,$$
taking also into account that $\omega_{Y| L_j}=\OO_{L_j}(1)$, we compute that $h^0(\cY, \omega_{\cY})=12$, and hence $\chi(\cY, \OO_{\cY})=13.$
The morphism $\tilde{\gamma}:=\gamma\circ \nu_Y:\cY\rightarrow \PP^1$ is a family of Prym curves of genus $6$ and it induces a moduli map $m(\tilde{\gamma}):\PP^1\rightarrow \rr_6$.

The points $u_1, \ldots, u_4\in \PP^2$ being general, the curve $\mathfrak{e}:=m(\tilde{\gamma})(\PP^1) \subset \rr_6$ is disjoint from the pull-back $\pi^*(\overline{\mathcal{GP}}_6)\subset \rr_6$ of the Gieseker-Petri divisor consisting  of curves of genus $6$ lying on a singular quintic del Pezzo surface, see \cite{FGSMV} for details on the geometry of $\pi^{-1}(\overline{\mathcal{GP}}_6)$. Since $\pi^*\bigl([\overline{\mathcal{GP}}_6]\bigr)_{| \pr_6}=94\lambda-12(\delta_0^{'}+\delta_0^{''}+2\delta_0^{\mathrm{ram}})\in CH^1(\pr_6)$, and $\mathfrak e\cdot \delta_0^{'}=77$ (this being the already computed number of nodal conic bundles in $P$), whereas $\mathfrak e\cdot \delta_0^{''}=0$, we obtain the following relation
$$47 \mathfrak{e} \cdot \lambda-6\mathfrak{e}\cdot \delta_0^{'}-12\mathfrak{e}\cdot \delta_0^{\mathrm{ram}}=0.$$
Finally, we observe that $\mathfrak e\cdot \lambda=\chi(\cY, \OO_{\cY})+g-1=18$, which leads to $\mathfrak e\cdot \delta_0^{\mathrm{ram}}=32$.
\end{proof}

\section{A sweeping rational curve in the boundary of $\aa_6$}
In this section we construct an explicit sweeping rational curve in $\pc^5$, whose numerical properties we shall use in order to bound the slope of $\aa_6$.
Before doing that, we quickly review basic facts concerning  the moduli space $\rr_g$ of stable Prym curves of genus $g$, while referring  to \cite{FL} for details.

Geometric points of $\rr_g$ correspond to triples $(X, \eta, \beta)$, where $X$ is a quasi-stable curve of arithmetic genus $g$, $\eta$ is a line bundle on $X$ of degree $0$, such that $\eta_{E}=\OO_E(1)$ for each smooth rational component $E\subset X$ with $|E\cap \overline{(X-E)}|=2$, and $\beta:\eta^{\otimes 2}\rightarrow \OO_X$ is a sheaf homomorphism whose restriction to any non-exceptional component of $X$ is an isomorphism. Denoting by $\pi:\rr_g\rightarrow \mm_g$ the forgetful map, one has the following formula \cite{FL} Example 1.4
\begin{equation}\label{pullbackrg}
\pi^*(\delta_0)=\delta_0^{'}+\delta_0^{''}+2\delta_{0}^{\mathrm{ram}}\in CH^1(\rr_g),
\end{equation}
where $\delta_0^{'}:=[\Delta_0^{'}], \, \delta_0^{''}:=[\Delta_0^{''}]$, and $\delta_0^{\mathrm{ram}}:=[\Delta_0^{\mathrm{ram}}]$ are boundary divisor classes on $\rr_g$ whose meaning we recall. Let us fix a general point $[C]\in \Delta_0$ corresponding to a smooth $2$-pointed curve $(N, x, y)$ of genus $g-1$ with normalization map $\nu:N\rightarrow C$, where $\nu(x)=\nu(y)$. A general point of $\Delta_0^{'}$ (respectively of $\Delta_0^{''}$) corresponds to a stable Prym curve $[C, \eta]$, where $\eta\in \mbox{Pic}^0(C)[2]$ and $\nu^*(\eta)\in \mbox{Pic}^0(N)$ is non-trivial
(respectively, $\nu^*(\eta)=\OO_N$). A general point of $\Delta_{0}^{\mathrm{ram}}$ is of the form $(X, \eta)$, where $X:=N\cup_{\{x, y\}} \PP^1$ is a quasi-stable curve of arithmetic genus $g$, whereas $\eta\in \mbox{Pic}^0(X)$ is a line bundle characterized by $\eta_{\PP^1}=\OO_{\PP^1}(1)$ and $\eta_N^{\otimes 2}=\OO_N(-x-y)$.
Throughout this paper, we only work on the partial compactification $\pr_g:=\pi^{-1}(\cM_g\cup \Delta_0^0)$ of $\cR_g$, where $\Delta_0^0$ is the open subvariety of $\Delta_0$ consisting of irreducible one-nodal curves. We denote by $\delta_0^{'}, \delta_0^{''}$ and $\delta_0^{\mathrm{ram}}$ the restrictions of the corresponding boundary classes to $\pr_g$. Note that $CH^1(\pr_g)={\bf Q}\langle \lambda, \delta_0^{'}, \delta_0^{''}, \delta_0^{\mathrm{ram}}\rangle$.

 \vskip 4pt

Recall that we use the identification $\PP^{15}:= \bigl|\mathcal I^2_{\{w_1, \ldots, w_4\}}(2,2)\bigr| =|\OO_{\PP}(2)|$ for the linear system of $(2,2)$ threefolds in $\PP^2\times \PP^2$ which are nodal
at $w_1, \ldots, w_4$.  Recall also that $\PP\rightarrow S$ is the $\PP^2$-bundle constructed in section 2.

%Furthermore, we set  $U := S - \bigcup_{i=1}^4 E_i$ and $U':= \mathbf P^2-\lbrace w_1, \ldots, w_4 \rbrace$. Obviously, $\sigma:U\rightarrow U'$ is biregular, and we have an identification $\pi^{-1}(U) = U' \times \mathbf P^2$ via the biregular map $\epsilon: \pi^{-1}(U) \to U' \times \mathbf P^2$.
%First we construct concretely starting from $\mathbf P^2 \times \mathbf P^2$ a family of sweeping curves in $\pc^5$, then we switch to conic bundles in $\PP$.

\vskip 5pt

 We start constructing a sweeping curve $i:\PP^1\rightarrow \pc^5$,  by fixing general points $(o_1, \ell_1),\ldots, (o_4, \ell_4)\in   \PP^2 \times \mathbf (\PP^2)^{\vee}$ and
a general point $o\in\PP^2$.   We introduce the net
$$T:=\Bigl\{Q\in \PP^{15}: (o, o)\in Q \ \mbox{ and }  \{o_j\}\times \ell_i\subset Q \mbox{ for } j=1, \ldots, 4\Bigr\},$$
consisting of conic bundles containing the lines $\{o_1\}\times \ell_1, \ldots, \{o_4\}\times \ell_4$ and passing through the point $(o, o)\in \PP^2\times \PP^2$. Because of the genericity of our choices, the restriction
$$\mathrm{res}_{| \{o\}\times \PP^2}: T \rightarrow \bigl|\OO_{\{o\}\times \PP^2}(2)\bigr|$$ is an injective map and we can view $T$ as a general net of conics in $\PP^2$ passing through the fixed point $o\in \PP^2$. The discriminant curve of the net is a nodal cubic curve $\Delta_T \subset T$; its singularity corresponds to the only conic of type $\ell_0+m_0\in T$, consisting of a pair of lines $\ell_0$ and $m_0$ passing through $o$. \par To ease notation, we identify $\{o\}\times \PP^2$ with $\PP^2$ in everything that follows.  Denoting by $\PP^1:=\PP(T_o(\PP^2))$ the pencil of lines through $o$,  it is clear that the map
$$\tau:\PP^1\rightarrow \Delta_T, \ \ \mbox{ } \tau(\ell):=Q_{\ell}\in T, \mbox{ such that } Q_{\ell}\supset \{o\}\times \ell,$$ is the normalization map of $\Delta_T$.
In particular, we have $\tau(\ell_0)=\tau(m_0)=\ell_0+m_0$, where, abusing notation, we identify $Q_{\ell}$ with its singular conic $\{o\}\times (\ell+m)=Q_{\ell}\cdot \bigl(\{o\}\times \PP^2\bigr)$. For $\ell \in \mathbf P^1$, the double cover $f_{\ell}: \widetilde{\Gamma}_{\ell} \to \Gamma_{\ell}$ over the discriminant curve $\Gamma_{\ell}$ of $Q_{\ell}$ is an element of $\P_6$ (see Definition \ref{p6}). Clearly $\widetilde{\Gamma}_{\ell}$  carries the marked points $\ell_1, \ell_2, \ell_3, \ell_4$ and $\ell$. This procedure induces a moduli map into the universal symmetric product
$$ i: \mathbf P^1 \to \pc^5, \ \mbox{ } \ i(\ell):=\bigl[\rho(\widetilde{\Gamma}_{\ell}/\Gamma_{\ell}),\  \ell_1,\  \ell_2,\  \ell_3,\ \ell_4,\ \ell\bigr].$$

We explicitly construct the family of discriminant curves $\Gamma_{\ell}$ of the conic bundles $Q_\ell$, where $\tau(\ell) \in \Delta_T$. Setting coordinates $x:=[x_1:x_2:x_3], y:=[y_1:y_2:y_3]$ in $\PP^2$,
let $$Z:=\Bigl\{(x, y, t)\in \PP^2\times \PP^2\times \Delta_{T}: y\in \mathrm{Sing}\bigl(\pi_1^{-1}(x)\cap Q_t\bigr) \Bigr\}\subset \PP^2\times \PP^2\times T.$$
Concretely, if $Q_1, Q_2, Q_3$ is a basis of $T$, then the surface $Z$ is given by the equations
$$\frac{\partial}{\partial y_i}\Bigl(t_1 Q_1(x, y)+ t_2 Q_2(x,y)+ t_3 Q_3(x, y)\Bigr)=0, \  \ \mbox{ for } i=1,2,3,$$
where $[t_1: t_2: t_3]\in \PP^2$ gives rise to the point $t\in T$, once the basis $Q_1, Q_2, Q_3$ of $T$ has been chosen. It follows immediately that $Z$ is a complete intersection of three divisors of multidegree $(2, 1, 1)$, defined by the partial derivatives,  and the divisor $\mathbf P^2 \times \mathbf P^2 \times \Delta_T$ of multidegree $(0, 0, 3)$.

\begin{lemma}
The first projection $\gamma_1:Z\rightarrow \PP^2$ is a map of degree $9$.
\end{lemma}
\begin{proof}
Denoting by $h_1, h_2, h_3\in \mbox{Pic}(\PP^2\times \PP^2\times T)$ the pull-backs of the hyperplane bundles from the three factors, we find that
$\mbox{deg}(\gamma_1)=(2h_1+h_2+h_3)^3\cdot (3h_3)\cdot h_1^2=9$.
\end{proof}

The third projection $\gamma_3: Z \to \Delta_T$ is a birational model of the family $\lbrace C_{\ell}\}_{\ell \in \mathbf P^1}$  of underlying genus $6$ curves, induced by the map $i:\PP^1\rightarrow \pc^5$. However, the surface $Z$ is not normal. It has singularities along the curves $\{w_j\}\times \Delta_T$ for $j=1, \ldots, 4$, as well as along the fibre $\gamma_3^{-1}(\ell_0+m_0)$ over the point $\tau(\ell_0) = \tau(m_0)=\ell_0+m_0\in T$.  To construct a smooth model of $Z$, we pass instead to its natural birational model in the $5$-fold $\PP \times \mathbf P^1$.

\vskip 3pt

Abusing notation, we still denote by $Q_{\ell}\subset \PP$ the strict transform of the conic bundle $Q_{\ell}$ in $\mathbf P^2 \times \mathbf P^2$; its discriminant curve
 $C_{\ell}$ is viewed an an element of $|-K_S |$. We denote by $\pi_{\ell}:Q_{\ell}\rightarrow S$ the restriction of $\pi:\PP\rightarrow S$, then consider the surface
$$\cZ:=\Bigl\{(z, \ell)\in \PP\times \PP^1:  z\in \mathrm{Sing} \ \pi_{\ell}^{-1}(C_{\ell})\Bigr\}. $$
Clearly $\mathcal Z$ is  endowed with the projection $q_{\mathbf P^1}: \mathcal Z \to \mathbf P^1$. We have the following commutative  diagram, where $u:=\bigl(\sigma \times \mathrm{id}_{\mathbf P^2}\bigr) \circ \epsilon^{-1}:\PP \to \mathbf P^2 \times \mathbf P^2$ and the horizontal arrows are the natural inclusions or projections and $\nu_Z:\mathcal{Z}\rightarrow Z$ is the normalization map:

$$
     \xymatrix{
         \mathcal{Z} \ar[r] \ar@/^2pc/[rr]^{q_{\PP^1}} \ar[d]_{\nu_Z}  & \PP\times \PP^1 \ar[d]^{u\times \tau} \ar[r] & \PP^1 \ar[d]_{\tau} \\
          Z \ar[r]^{} \ar@/_2pc/[rr]^{\gamma_3}      & \PP^2\times \PP^2\times \Delta_T \ar[r]^{} & \Delta_T }
$$

Since $u \times \tau$ is birational, it follows that $\mbox{deg}(\cZ/S)=\mbox{deg}(Z/\PP^2)=9$.  The fibration $q_{\mathbf P^1}: \mathcal Z \to \mathbf P^1$ admits sections $$ \tau_j:\PP^1\rightarrow \cZ\  \mbox{ for } j = 1, \ldots, 5,$$
 which we now define. For $1\leq j\leq 4$ and each $\ell\in \PP^1$,  the fibre  $Q_{\ell} \cdot \bigl(\{o_j\}\times \PP^2\bigr)$ contains the line $\ell_j$. Hence $\ell_j$ defines a point in the covering curve of $f_{\ell}: \widetilde C_\ell \to C_\ell$.
By definition $\tau_j(\ell)$ is this point. Tautologically, $\tau_5(\ell)$ is the point corresponding to the line $\ell$. \par
Finally, we consider the universal family $\mathcal Q \subset \PP \times T$ defined by $T$. The pull-back of the projection $\mathcal Q \to T$ by the morphism $\mathrm{id}_{\PP} \times \tau$ induces a flat family of conic bundles $\mathcal Q' \subset \PP \times \mathbf P^1$ and a  projection $q':\mathcal Q' \to \PP^1$. Clearly, $\mathcal Z \subset \mathcal Q'$ and $q_{\PP^1} = q'_{|\mathcal Z}$.

\begin{definition}
A conic bundle $Q\in |\OO_{\PP}(2)|$ is said to be \emph{ordinary} if both $Q$ and its discriminant cover curve $C$ are nodal. A subvariety in $|\OO_{\PP}(2)|$ is said to be a  \emph{Lefschetz family},  if each of its members is an ordinary conic bundles.
\end{definition}
Postponing the proof, we assume that the fibration $q':\mathcal Q\rightarrow \PP^1$ constructed above is a Lefschetz family of conic bundles, and we determine the properties of the Prym moduli map $m: \mathbf P^1 \to \overline {\mathcal R}_6$, where $m(\ell):=[f_{\ell}: \widetilde C_{\ell} \to C_{\ell}]=\rho\bigl(\widetilde{\Gamma}_{\ell}/\Gamma_{\ell}\bigr)$.

\begin{proposition}\label{pencil2}  The numerical features of $m:\PP^1\rightarrow \widetilde{\cR}_6\subset \rr_6$ are as follows:
$$m(\PP^1)\cdot \lambda=9\cdot 6, \ m(\PP^1)\cdot \delta_{0}^{'}=3\cdot 77, \ m(\PP^1)\cdot \delta_0^{\mathrm{ram}}=3\cdot 32, \ m(\PP^1)\cdot \delta_0^{''}=0.$$
\end{proposition}
\begin{proof}
We consider the composite map $\rho \circ \mathfrak{d}_{| T}: T\dashrightarrow \rr_6$, assigning to a conic bundle from the net $T\subset \PP^{15}$ the double covering of its (normalized) discriminant curve. This map is well-defined outside the codimension two locus in $T$ corresponding to conic bundles with non-nodal discriminant. Furthermore, $m=\rho \circ \mathfrak{d} \circ \tau:\PP^1\rightarrow \rr_6$, where we recall that $\tau(\PP^1)=\Delta_T\subset T$ is a nodal cubic curve. It follows that the intersection number of $m(\PP^1)\subset \rr_6$ with
any divisor class on $\rr_6$ is three times the intersection number of the corresponding class in $CH^1(\rr_6)$ with the curve of discriminants induced by a \emph{pencil} of conic bundles in $|\OO_{\PP}(2)|$. The latter numbers  have been determined in Theorem  \ref{singconic}.
\end{proof}
\vskip 3pt

\begin{definition}
An irreducible variety $X$ is said to be swept by an irreducible curve $\Gamma$ on $X$, if $\Gamma$ flatly deforms in a family of curves $\{\Gamma_t\}_{t\in T}$ on $X$ such that for a general point $x\in X$, there exists $t\in T$ with $x\in \Gamma_t$.
\end{definition}

The composition of the map $i:\PP^1\rightarrow \pc^5$ with the projection $\pc^5\rightarrow \rr_6$ is the map $m:\PP^1\rightarrow \rr_6$ discussed in Proposition \ref{pencil2}.
We discuss the numerical properties of $i$:

 \begin{proposition}\label{pencil3}
The moduli map $i:\PP^1\rightarrow \pc^5$ induced by the pointed family of Prym curves
$$\bigl(q_{\PP^1}:\cZ\rightarrow \PP^1, \tau_1, \ldots, \tau_5:\PP^1\rightarrow \cZ\bigr)$$ sweeps the five-fold product $\pc^5$. Furthermore $i(\PP^1)\cdot \psi_{x_j}=9$, for $j=1, \ldots, 5$.
\end{proposition}
\begin{proof}
For $1\leq j\leq 4$, the image of the section $\tilde{\tau}_j:=\nu_Z\circ \tau_j:\PP^1\rightarrow Z$ is the curve
$$L_j:=\bigl\{(o_j, \ y_j(\ell),\ \nu(\ell))\in \PP^2\times \PP^2\times T:\ell\in \PP^1\bigr\},$$
where $y_j(\ell)=\ell_j \cap m_j(\ell)$, with $m_j(\ell)$ being the line in $\PP^2$ defined by the equality of cycles $Q_{\ell} \cdot \bigl(\{o_j\}\times \PP^2 \bigr)=\{o_j\} \times (\ell_j+m_j(\ell))$. Here, recall that $\ell\in \PP^1=\PP(T_o(\PP^2)$ is a point corresponding to a line in $\PP^2$ passing through $o$.
In particular, noting that by the adjunction formula $\omega_Z=\OO_Z(3h_1+3h_3)$, we compute that $L_j\cdot h_1=0$ and $L_j\cdot h_3=3$, hence $L_j\cdot \omega_Z=L_j\cdot (3h_1+3h_3)=9$.

By definition $i(\PP^1)\cdot \psi_{x_j}=\tau_j^*\bigl(c_1(\omega_{q_{\PP^1}})\bigr)$. To evaluate the dualizing class, we note that $\omega_{q_{\PP^1}}=\omega_{\cZ}\otimes q_{\PP^1}^*(T_{\PP^1})$, therefore $\mbox{deg } \tau_j^*\bigl(c_1(\omega_{q_{\PP^1}})\bigr)=\mbox{deg } \tau_j^*(\omega_{\mathcal{Z}})+2$. Furthermore,
$$\nu_Z^*(\omega_Z)=\omega_{\mathcal{Z}}\otimes \OO_{\cZ}\bigl(q_{\PP^1}^{-1}(\ell_0)+q_{\PP^1}^{-1}(m_0)+D\bigr),$$
where $D\subset \cZ$ is a curve disjoint from $\nu_Z^{-1}(L_j)$. We compute that
$$\mbox{deg } \tau_j^*(\omega_{\mathcal{Z}})=\mbox{deg } \tilde{\tau}_j^*(\omega_Z)-\mbox{deg }\tau_j^* q_{\PP^1}^*(\ell_0)-\mbox{deg }\tau_j^* q_{\PP^1}^*(m_0)=\omega_Z\cdot L_j-2,$$
and finally, $i(\PP_1)\cdot \psi_{x_j}=\omega_Z\cdot L_j=9$. The calculation of $i(\PP^1)\cdot \psi_{x_5}$ is largely similar and we skip it.
\end{proof}

\vskip 4pt

 \begin{proof}[Proof of the claim] We show that $q':\mathcal{Q}\rightarrow \PP^1$ is a Lefschetz family, that is, it consists entirely of ordinary conic bundles.
 For $1\leq j\leq 4$, let $\ell_j'\subset \PP$ be the inverse image of the line
 $\{o_j\}\times \ell_j$ under the map $u:\PP\dashrightarrow \PP^2\times \PP^2$  and set $W := \bigl|\mathcal I_{\{\ell_1', \ldots, \ell_4'\}}(2)\bigr|\subset |\OO_{\PP}(2)|$. The net
$T:=T_o$ of conic bundles passing through the point $(o, o)\in \PP^2\times \PP^2$ is a plane in $W$. Let
$\Delta_{\mathrm{no}}$ denote the locus of non-ordinary conic bundles $Q \in W$. We aim to show that $\Delta_{\mathrm{no}} \cap \Delta_{T_o} = \emptyset$, for a general point $o\in \PP^2$.

We consider the incidence correspondence
$$\Sigma:=\Bigl\{\bigl(Q, (o, \ell)\bigr)\in \Delta_{\mathrm{no}}\times \GG: \{o\}\times \ell\subset u(Q), \ \ o\in \ell\Bigr\}$$
together with the projection map $p_1: \Sigma\rightarrow \Delta_{\mathrm{no}}$.
Over a conic bundle $Q \in \Delta_{\mathrm{no}}$ for which the image $u(Q) \subset \PP^2\times \PP^2$ is transversal to a general fibre $\{o\}\times \PP^2$, the fibre $p_1^{-1}(Q)$ is finite. To account for the conic bundles not enjoying this property, we define $\Delta_{\mathrm{hr}}$ to be the union of the irreducible components of $\Delta_{\mathrm{no}}$ consisting of conic bundles $Q\in W$ such that the branch locus of $Q\rightarrow S$ is equal to $S$.

To conclude that $\Delta_{\mathrm{no}}\cap \Delta_{T_o}=\emptyset$ for a general $o\in \PP^2$, it suffices to show that
(1) $\Delta_{\mathrm{no}}$ has codimension at least $2$ in $W$, and (2) $\Delta_{\mathrm{hr}}$ has codimension at least $3$ in $W$.
The next two lemmas are devoted to the proof of these assertions.
 \end{proof}

 \begin{lemma} $\Delta_{\mathrm{no}}$ has codimension at least $2$ in $W$.
 \end{lemma}
 \begin{proof}  We have established that $h: \PP \to \mathbf P^4$ is a morphism of degree two. We claim that the $4$ lines $l_i:=h(\ell_i')\subset \PP^4$ are general, in the sense that  $V:=\bigl|\I_{\{l_1, \ldots, l_4\}}(2)\bigr|$ is a net of quadrics. Granting this and denoting by $L_{ij}\in H^0(\PP^4, \OO_{\PP^4}(1))$ the linear form vanishing along $l_i \cup l_j$, the space $V$ is
 generated by the quadrics $L_{12}\cdot L_{34}, L_{13}\cdot L_{24}$ and $L_{14}\cdot L_{23}$ respectively. The base locus $\mbox{bs } |V|$ of the net is a degenerate canonical curve of genus $5$,  which is a union of $8$ lines, namely $l_1, \ldots, l_4$ and $b_1, \ldots, b_4$, where if $\{1, 2, 3, 4\}=\{i, j, k, l\}$, then the line $b_l\subset \PP^4$ is the common transversal to the lines $l_i, l_j$ and $l_k$. Then by direct calculation, the pull-back $P$ of a general pencil in $V$ is a Lefschetz family of conic budles in $|\OO_{\PP}(2)|$. Since $P \cap \Delta_{\mathrm{no}} = \emptyset$, it follows that $\mbox{codim}(\Delta_{\mathrm{no}}, W) \geq 2$. It remains to show that the lines $l_1, \ldots, l_4$ are general. To that end, we observe that the construction can be reversed. Four general lines $m_1, \ldots, m_4\in \GG(1, 4)\subset \PP^9$  define a codimension $4$ linear section $S'$ of $\GG(1,4)$ which is isomorphic to $S$. The projectivized universal bundle $\PP' \to S'$ is a copy of $\PP$ and the projection $h': \PP' \to \mathbf P^4$ is the tautological map.  This completes the proof.
 \end{proof}

  The second lemma follows from a direct analysis in $\PP^2 \times \PP^2$.

\begin{lemma} $\Delta_{\mathrm{hr}}$ has codimension at least $3$ in $W$.
\end{lemma}

\begin{proof} If $Q$ is a general element of an irreducible component of $\Delta_{\mathrm{hr}}$ , then the discriminant locus of the projection $p:Q\rightarrow S$ equals $S$, and necessarily $Q = D + D'$, where $p(D)= p(D') = S$. By a dimension count, it follows that $W$ contains only \emph{finitely many} elements $Q \in \Delta_{\mathrm{hr}}$, such that $D, D' \in |\OO_{\PP}(1)|$, and assume that we are not in this case.

Recall that $h': \mathbf P^2 \times \mathbf P^2 \dashrightarrow \mathbf P^4$ is the map defined by
$\bigl|\mathcal I_{\{w_1, \ldots,  w_4\}}(1,1)\bigr|$. The case when both $u(D), u(D') \in |\mathcal I_{\{w_1, \ldots, w_4\}}(1,1)|$ having been excluded, we may assume that one of the components of $u(Q)$, say $u(D)\subset \PP^2\times \PP^2$, has type $(0,1)$. In particular, $u(D) = \mathbf P^2 \times n$, where $n\subset \PP^2$ is a line. Observe that $u(D)$ has degree three in the Segre embedding $\mathbf P^2 \times \mathbf P^2 \subset \mathbf P^8$ and the base scheme of $\bigl|\mathcal I_{\{w_1, \ldots, w_4\}}(1,1)\bigr|$ consists of the simple points $w_1, \ldots, w_4$.  Since $h'(D)$ lies on a quadric, it follows $ u(D) \cap \lbrace w_1, \ldots, w_4 \rbrace \neq \emptyset$, therefore we have $u_i \in n$ for some $i$, say $i = 4$. Since the lines $\lbrace o_i\rbrace \times \ell_i$ are general, they do not lie on $u(D)$, for $\ell_i\neq n$.  Hence $u(D')\subset \PP^2\times \PP^2$
is a $(2,1)$ hypersurface which contains $\lbrace o_1 \rbrace \times \ell_1, \ldots \{o_4\}\times \ell_4$, is singular at $w_1, w_2, w_3$ and such that $w_4 \in u(D')$. This contradicts the generality of the lines $\{o_i\}\times \ell_i$.
\end{proof}

\section{The slope of $\aa_6$}
For $g\geq 2$, let $\aa_g$ be
the first Voronoi compactification of $\cA_g$ --- this is the toroidal compactification of $\cA_g$ constructed using the perfect fan decomposition, see \cite{SB}. The rational Picard group of $\aa_g$ has rank $2$  and it is generated by the first Chern class $\lambda_1$ of the Hodge bundle and the class of the \emph{irreducible} boundary divisor $D=D_g:=\aa_g-\cA_g$. Following Mumford \cite{Mu}, we consider the moduli space $\pa_g$ of principally polarized abelian varieties of dimension $g$ together with their rank $1$ degenerations. Precisely, if $\xi:
\aa_g\rightarrow \cA_g^s=\cA_g\sqcup \cA_{g-1}\sqcup \ldots \sqcup \cA_1\sqcup \cA_0$ is the projection from the toroidal to the Satake compactification of $\cA_g$, then $$\pa_g:=\aa_g-\xi^{-1}\Bigl(\bigcup_{j=2}^g \cA_{g-j}\Bigr):=\cA_g \sqcup \widetilde{D}_g,$$
where $\widetilde{D}_g$ is an open dense subvariety of $D_g$ isomorphic to the universal Kummer variety $\mbox{Kum}(\cX_{g-1}):=\cX_{g-1}/\pm$. Furthermore, if $\phi:\px_{g-1}\rightarrow \pa_{g-1}$ is the extended universal abelian variety, there exists a degree two morphism $j:\px_{g-1}\rightarrow \aa_g$, extending the Kummer map $\px_{g-1}\stackrel{2:1}\rightarrow \widetilde{D}_g$. The geometry of the boundary divisor $\partial \px_{g-1}=\phi^{-1}(\mbox{Kum}(\cX_{g-2}))$ is discussed in \cite{vdG} and \cite{EGH}. In particular, $\mbox{codim}(j(\partial \px_{g-1}), \aa_g)=2$. As usual, let $\mathbb E_g$ denote the Hodge bundle on $\aa_g$.

Denoting by $\varphi:\py_g\rightarrow \pr_g$ the universal Prym variety restricted to the partial compactification $\pr_g$ of $\rr_g$ introduced in Section 3, we have the following commutative diagram summarizing the situation, where the lower horizontal arrow is the Prym map:
$$
     \xymatrix{
         \py_g \ar[r]^{\chi} \ar[d]_{\varphi} & \px_{g-1} \ar[d]_{\phi} \ar[r]^{j} & \aa_g \\
          \pr_g \ar[r]^{P}       & \pa_{g-1} }
$$

Furthermore, let us denote by $\theta\in CH^1(\px_{g-1})$ the class of the universal theta divisor trivialized along the zero section and by $\theta_{\mathrm{pr}}:=\chi^*(\theta)\in CH^1(\py_g)$ the Prym theta divisor. The
following formulas have been pointed out to us by Sam Grushevsky:
\begin{proposition}\label{mumford}
The following relations at the level of divisor classes hold:
\begin{enumerate}
\item $j^*([D])= -2\theta+\phi^*([D_{g-1}])\in CH^1(\px_{g-1})$.
\item $(j\circ \chi)^*(\lambda_1)=\varphi^*\bigl(\lambda-\frac{1}{4}\delta_0^{\mathrm{ram}}\bigr)\in CH^1(\py_g).$
\item $(j\circ \chi)^*([D])=-2\theta_{\mathrm{pr}}+\varphi^*(\delta_0^{'})\in CH^1(\py_g).$
\end{enumerate}
\end{proposition}
\begin{proof}
At the level of the restriction $j:\cX_{g-1}\rightarrow \aa_g$, the formula
$$j^*(D)\equiv -2\theta \in CH^1(\cX_{g-1})$$ is proven in \cite{Mu} Proposition 1.8. To extend this calculation to $\px_{g-1}$, it suffices to observe that the boundary divisor $\partial \px_{g-1}=\phi^*(\widetilde{D}_{g-1})$ is mapped under $j$ to the locus in $\aa_g$ parametrizing \emph{rank $2$} degenerations and it will appear with multiplicity one in $j^*(D)$.

To establish relation (ii), we observe that $j^*(\lambda_1)=\phi^*(\lambda_1)$, where we use the same symbol to denote the Hodge class on $\cA_g$ and that on $\cA_{g-1}$. Indeed, there exists an exact sequence of vector bundles on $\px_g$, see also \cite{vdG} p.74:
$$0\longrightarrow \phi^*(\mathbb E_{g-1})\longrightarrow j^*(\mathbb E_g)\longrightarrow \OO_{\px_{g-1}}\longrightarrow 0.$$
It follows that $\chi^*j^*(\lambda_1)=\varphi^* P^*(\lambda_1)=\varphi^*(\lambda-\frac{1}{4}\delta_0^{\mathrm{ram}})$, where $P^*(\lambda_1)=\lambda-\frac{1}{4}\delta_0^{\mathrm{ram}}$, see \cite{FL}, \cite{GSMH}. Finally, (iii) is a consequence of (i) and of the relation  $P^*([\widetilde{D}_{g-1}])=\delta_0^{'}$, see \cite{GSMH}.
\end{proof}

Assume now that $g$ is an even integer and let $\tilde{\pi}:\pc\rightarrow \rr_g$ be the universal curve of genus $2g-1$, that is, $\pc=\mm_{2g-1,1}\times_{\mm_{2g-1}}\rr_g$, and
$\pi:\cc\rightarrow \rr_g$ the universal curve of genus $g$, that is, $\cc=\mm_{g,1}\times_{\mm_g} \rr_g$. There is a degree two map $f:\pc\rightarrow \cc$  unramified in codimension one and an involution $\iota:\pc\rightarrow \pc$, such that $f\circ \iota=f$. Note that $\omega_{\tilde{\pi}}=f^*(\omega_{\pi})$.

\vskip 3pt
We consider the global Abel-Prym map
$\mathfrak{ap}:\pc^{g-1}\dashrightarrow \py_g$, defined by
$$\mathfrak{ap}\Bigl(\widetilde{C}/C, \ x_1, \ldots, x_{g-1}\Bigr):=\Bigl(\widetilde{C}/C, \ \OO_{\widetilde{C}}\bigl(x_1-\iota(x_1)+\cdots+x_{g-2}-\iota(x_{g-2})+2x_{g-1}-2\iota(x_{g-1})\bigr)\Bigr).$$

%For odd $g$, one has even a symmetric Abel-Prym map given by
%$$\Bigl(\widetilde{C}/C, x_1, \ldots, x_{g-1}\Bigr)\mapsto \Bigl(\widetilde{C}/C, \ \OO_{\widetilde{C}}\bigl(x_1-\iota(x_1)+\cdots+x_{g-1}-\iota(x_{g-1})\bigr)\Bigr).$$

\begin{remark}
We recall that if $\widetilde{C}\rightarrow C$ is an \'etale double cover and $\iota:\widetilde{C}\rightarrow \widetilde{C}$ the induced involution, then the Prym variety $P(\widetilde{C}/C)\subset \mbox{Pic}^0(\widetilde{C})$ can be realized as the
locus of line bundles $\OO_{\widetilde{C}}(E-\iota(E))$, where $E$ is a divisor on $\widetilde{C}$ having \emph{even} degree, see \cite{B3}. Furthermore, for a general point
$[\widetilde{C}\rightarrow C]\in \rr_g$, where $g\geq 3$, and for an integer $1\leq n\leq g-1$, the difference map $\widetilde{C}_n\rightarrow \mbox{Pic}^0(\widetilde{C})$ given by $E\mapsto \OO_{\widetilde{C}}(E-\iota(E))$ is generically finite.
In particular, for even $g$, the locus
$$Z_{g-2}(\widetilde{C}/C):=\Bigl\{\OO_{\widetilde{C}}(E-\iota(E)): E\in \widetilde{C}_{g-2}\Bigr\}$$ is a divisor inside $P(\widetilde{C}/C)$. We refer to $Z_{g-2}(\widetilde{C}/C)$ as the \emph{top difference Prym variety}.
\end{remark}

\vskip 4pt

One computes the pull-back of the universal theta divisor under the Abel-Prym map. Recall that $\psi_{x_1}, \ldots, \psi_{x_{g-1}}\in CH^1(\pc^{g-1})$ are the cotangent classes corresponding to the marked points on the curves of genus $2g-1$.
\begin{proposition}\label{ap}
For even $g$, if $\mu=\varphi\circ \mathfrak{ap}:\pc^{g-1}\rightarrow \pr_g$ denotes the projection map, one has
$$\mathfrak{ap}^*(\theta_{\mathrm{pr}})=\frac{1}{2}\sum_{j=1}^{g-2} \psi_{x_j}+2\psi_{x_{g-1}}+0\cdot\Bigl(\lambda+\mu^*(\delta_0^{'}+\delta_0^{''}+\delta_0^{\mathrm{ram}})\Bigr)-\cdots\in CH^1(\pc^{g-1}).$$
\end{proposition}
\begin{proof} We factor the map $\mathfrak{ap}:\pc^{g-1}\dashrightarrow \py_g$ as $\mathfrak{ap}=\mathfrak{aj}\circ \Delta$, where $\Delta:\pc^{g-1}\rightarrow \pc^{2g-2}$ is defined by $(x_1, \ldots, x_{g-1})\mapsto (x_1, \ldots, x_{g-1}, \iota(x_1), \ldots, \iota(x_{g-1}))$ and $\mathfrak{aj}:\pc^{2g-2}\dashrightarrow \overline{\mathfrak{Pic}}_{2g-1}^0$ is the difference Abel-Jacobi map between the first and the last $g-1$ marked points on each curve into the universal Jacobian of degree zero over $\mm_{2g-1}$ respectively . There is a generically injective rational map $\py_g\dashrightarrow \overline{\mathfrak{Pic}}_{2g-1}^0$, which globalizes the usual inclusion $P(\widetilde{C}/C)\subset \mbox{Pic}^0(\widetilde{C})$ valid for each Prym curve $[\widetilde{C}\rightarrow C]\in \cR_g$. Using \cite{GZ} Theorem 6, one computes the pull-back $\mathfrak{aj}^*(\theta_{2g-1})\in CH^1(\pc^{2g-2})$ of the universal theta divisor $\theta_{2g-1}$ on $\overline{\mathfrak{Pic}}_{2g-1}^0$ trivialized along the zero section. Remarkably, the coefficient of $\lambda$, as well as that of the $\delta_0^{'}, \delta_0^{''}$ and $\delta_0^{\mathrm{ram}}$ classes in this expression, are all zero.  This is then pulled-back to
$\pc^{g-1}$ keeping in mind that the pull-back of $\theta_{2g-1}$ to $\py_g$ is equal to $2\theta_{\mathrm{pr}}$. Using the  formulas $\Delta^*(\psi_{x_j})=\Delta^*(\psi_{y_j})=\psi_{x_j}$, and $\Delta^*(\delta_{0:x_i y_j})=\delta_{0: x_i x_j}$, as well as $\Delta^*(\delta_{0: y_i y_j})=\delta_{0: x_i x_j}$, we conclude.
\end{proof}

\begin{remark}
The other boundary coefficients of $\mathfrak{ap}^*(\theta_{\mathrm{pr}})\in CH^1(\pc^{g-1})$ can be determined explicitly, but play no role in our future considerations.
\end{remark}
\begin{remark}\label{restrap}
Restricting ourselves to even $g$, we consider the restricted (non-dominant) Abel-Prym map $\mathfrak{ap}_{g-2}:\pc^{g-2}\dashrightarrow \widetilde{\mathcal{Y}}_g$ given by
$$\mathfrak{ap}_{g-2}\bigl(\widetilde{C}/C,\  x_1, \ldots, x_{g-2}\bigr):=\Bigl(\widetilde{C}/C,\  \OO_{\widetilde{C}}\bigl((x_1-\iota(x_1)+\cdots +x_{g-2}-\iota(x_{g-2})\bigr)\Bigr),$$
and obtain the formula:
$\mathfrak{ap}_{g-2}^*(\theta_{\mathrm{pr}})=\frac{1}{2}\sum_{j=1}^{g-2} \psi_{x_j}+0\cdot \bigl(\lambda +\mu^*(\delta_0^{'}+\delta_0^{''}+\delta_0^{\mathrm{ram}})\bigr)-\cdots.$

The image of $\mathfrak{ap}_{g-2}$ is a divisor $\cZ_{g-2}$ on $\widetilde{\mathcal{Y}}_{g}$ characterized by the property
$$(\cZ_{g-2})_{| P(\widetilde{C}/C)}=Z_{g-2}(\widetilde{C}/C),$$ for each $[\widetilde{C}\rightarrow C]\in \cR_g$. In other words, $\cZ_{g-2}$ is the divisor cutting out on each Prym variety the top difference variety. A similar difference variety inside the universal Jacobian over $\mm_g$ has been studied in \cite{FV}. Specializing to the case $g=6$, the locus
$$\cU_{4}:=\overline{(j\circ \chi)(\cZ_{4})}\subset \aa_6$$ is a codimension two cycle on $\aa_6$, which will appear as an obstruction for an effective divisor on $\aa_6$ to have small slope.
\end{remark}

\vskip 2pt

We use these considerations to bound from below the slope of $\aa_6$.

\vskip 3pt

\noindent {\emph{Proof of Theorem \ref{slope}.}
We have seen that the boundary divisor $D_6$ of $\aa_6$ is filled-up by rational curves $h:\PP^1\rightarrow D_6$ constructed in Theorem \ref{pencil3} by pushing-forward the sweeping rational curve $i:\PP^1\rightarrow \pc^5$ of discriminants of a pencil of conic bundles. In particular,
$\gamma:=h_*(\PP^1)\in NE_1(\aa_6)$ is an effective class that intersects every non-boundary effective divisor on $\aa_6$ non-negatively. We compute using Propositions \ref{mumford} and \ref{ap}:
$$\gamma\cdot \lambda_1=i_*(\PP^1)\cdot \mu^*\Bigl(\lambda-\frac{1}{4}\delta_0^{\mathrm{ram}}\Bigr)=6\cdot 9-\frac{3\cdot 32}{4}=30, \ \ \mbox{ and }$$
$$\gamma\cdot [D_6]=-i_*(\PP^1)\cdot \Bigl(\sum_{j=1}^4 \psi_{x_j}+4\psi_{x_5}\Bigr)+ i_*(\PP^1)\cdot \mu^*(\delta_0^{'})=-8\cdot 9+3\cdot 77=159.$$
We obtain the bound $s(\aa_6)\geq \frac{\gamma\cdot [D_6]}{\gamma\cdot \lambda_1}=\frac{53}{10}$.
\hfill $\Box$
\vskip 4pt

For effective divisors on $\aa_6$ transversal to $\cU_4$, we obtain a better slope bound:

\begin{theorem}\label{u4}
If $E$ is an effective divisor on $\aa_6$ not containing the universal codimension two Prym difference variety $\cU_4\subset \aa_6$, then $s(E)\geq \frac{13}{2}$.
\end{theorem}
\begin{proof} We consider the family $(q_{\PP^1}:\cZ\rightarrow \PP^1, \ \tau_1, \ldots, \tau_4:\PP^1\rightarrow \cZ)$ obtained from the construction explained in Theorem \ref{pencil3}, where we retain only the first four sections. We obtain an induced moduli map $i_4:\PP^1\rightarrow \pc^4$. Pushing $i_4$ forward via the Abel-Prym map, we obtain a curve $h_4:\PP^1\rightarrow \cU_4\subset \aa_6$, which fills-up the locus $\cU_4$. Thus $\gamma_4:=(h_4)_*(\PP^1)\in NE_1(\aa_6)$ is an effective class which intersects non-negatively any effective divisor on $\aa_6$ not containing $\cU_4$. We compute using Theorems \ref{pencil2} and \ref{pencil3}:
$$\gamma_4\cdot \lambda_1=\gamma\cdot \lambda_1=30 \  \mbox{  }  \mbox{ and } \ \  \gamma_4\cdot [D_6]=-4\cdot 9+3\cdot 77=195.$$
\end{proof}

\end{document}